\pdfoutput=1
\documentclass[11pt]{article}

\usepackage[T1]{fontenc}
\usepackage{amsmath,amssymb,amsthm,booktabs,mathtools}
\usepackage[margin=0.92in]{geometry}
\usepackage[hidelinks,pdfusetitle]{hyperref}
\usepackage[nameinlink,noabbrev,capitalise]{cleveref}
\usepackage{microtype}
\usepackage{thmtools}
\declaretheorem[name=Theorem,numberwithin=section]{theorem}
\declaretheorem[name=Proposition,sibling=theorem]{proposition}
\declaretheorem[name=Lemma,sibling=theorem]{lemma}

\declaretheorem[name=Definition,style=definition,sibling=theorem]{definition}
\declaretheorem[name=Remark,style=remark,sibling=theorem]{remark}

\title{Equilibria of Inverse-Square Repulsion on the Line Are Arithmetic Progressions}
\author{Alper Ferudun\thanks{\raggedright Mercury Software GmbH.
Email: \href{mailto:alper@mercurycodelab.com}{alper@mercurycodelab.com}.
GitHub: \url{https://github.com/AlperTheKing}.}}
\date{October 3, 2026}

\hypersetup{
  pdfauthor={Alper Ferudun},
  pdfkeywords={points in equilibrium, inverse-square force, Riesz interaction, arithmetic progression, crystallization, reversible random walk, electrical network, Doob transform, recurrence}
}

\newcommand{\R}{\mathbb{R}}
\newcommand{\Z}{\mathbb{Z}}
\newcommand{\Ec}{\mathcal{E}}
\DeclareMathOperator{\sgn}{sgn}

\begin{document}
\maketitle

\begin{abstract}
Benjamini asked whether every configuration of points on the real line
that is in equilibrium under the inverse-square repulsive force must be an
arithmetic progression.  Georgakopoulos and Kolountzakis proved this when
some gap between consecutive points has maximal or minimal length, and
described the general (aperiodic) case as open.  We show that the answer
is yes.  More generally, let $1<s\le2$, and let $X\subset\R$ be a locally
finite set with at least two points such that, for every $x\in X$, the
total force \mbox{$\sum_{y\in X\setminus\{x\}}|y-x|^{-s}$} is finite and
the net force \mbox{$\sum_{y\in X\setminus\{x\}}\sgn(y-x)\,|y-x|^{-s}$} is
zero.  Then $X$ is an arithmetic progression.  No a priori assumption on the gaps is needed.
Subtracting the equilibrium equations of two consecutive points shows that
the gaps $g_n$ form a positive harmonic function for an explicit reversible
random walk on $\Z$ with long-range jumps.  Equilibrium also bounds the
ratio of consecutive gaps, by $1.5386\ldots$ when $s=2$.  With this bound, an
energy estimate shows that the Doob transform of the walk by $g$ is
recurrent.  Since $1/g$ is a positive harmonic function of the transformed
walk, it is constant.
\end{abstract}

\section{Introduction}

In a one-page note of May 2015, I.~Benjamini considered a locally finite
configuration $(a_n)_{n\in\Z}$ of points on the real line, with a force
such as $F(x,y)=|x-y|^{-2}$ between any two of them, and called the system
in equilibrium if $\sum_{i\ne n}F(a_i,a_n)\sgn(a_i-a_n)=0$ for every $n$
\cite{Ben15}.  Arithmetic progressions are in equilibrium, and Question~0.1
of \cite{Ben15} asks whether, conversely, every equilibrium satisfies
$a_n=\alpha n+\beta$ for all $n$.  The same question is Question~5 in
Section~4 of the problem list of the 2015 Warwick workshop ``Random walks on
graphs and potential theory'' \cite{Syl15}.

Georgakopoulos and Kolountzakis \cite{GK17} studied the question for a
general strictly decreasing force law $F$.  Their Problem~1 asks whether
all distances between subsequent particles of an equilibrium must be equal.
They proved that this is the case if some gap between consecutive particles
has maximal or minimal length \cite[Theorem~1]{GK17}.  In particular,
periodic equilibria are arithmetic progressions, and finite equilibria on a
circle are equally spaced.  For $F(d)=d^{-2}$ and uniformly discrete
configurations, they showed that if the particles on a half-line are in
equilibrium, then their positions determine the positions of all other
particles \cite[Theorem~4]{GK17}.  They also noted that the ratio of
consecutive gaps is bounded \cite[Proposition~1]{GK17}.  Finally, for every
continuous strictly decreasing $F$ they constructed non-trivial
configurations in which every particle except one is in equilibrium
\cite[Theorem~7]{GK17}.  For aperiodic configurations they describe the
problem as open, even for $F(d)=d^{-2}$, and they state that it is open for
every strictly decreasing $F$.

We answer the question for $F(d)=d^{-2}$, and more generally for
$F(d)=d^{-s}$ with $1<s\le2$.  The following definition makes the
conventions of \cite{GK17} explicit.

\begin{definition}\label{def:eq}
Let $s>1$.  An \emph{$s$-equilibrium} is a set $X\subset\R$ with at least
two points that is locally finite (every bounded interval contains only
finitely many points of $X$) and satisfies, for every $x\in X$,
\begin{gather}
\sum_{y\in X\setminus\{x\}}|y-x|^{-s}<\infty, \label{eq:E1}\\
\sum_{y\in X\setminus\{x\}}\sgn(y-x)\,|y-x|^{-s}=0. \label{eq:E2}
\end{gather}
\end{definition}

Condition \eqref{eq:E1} says that the total force exerted on each particle
from each side is finite.  It makes the series in \eqref{eq:E2} absolutely
convergent, so that \eqref{eq:E2} does not depend on an order of summation.
If \eqref{eq:E1} holds at one point $x_0\in X$, then every interval
$[x_0-M,x_0+M]$ contains at most $1+M^s\sum_{y\ne x_0}|y-x_0|^{-s}$ points
of $X$.  So \eqref{eq:E1} at a single point already implies that $X$ is
locally finite, and then \eqref{eq:E1} holds at every point.  We keep local
finiteness in the definition because it is part of Benjamini's
formulation.

\begin{theorem}\label{thm:main}
Let $1<s\le2$, and let $X$ be an $s$-equilibrium, that is, a locally finite
set $X\subset\R$ with at least two points, such that for every $x\in X$ the
total force \eqref{eq:E1} from both sides is finite and the net force
\eqref{eq:E2} is zero.  Then $X=\{\alpha n+\beta:n\in\Z\}$ for some
$\alpha>0$ and $\beta\in\R$.  Equivalently, if $(a_n)_{n\in\Z}$ is the
increasing enumeration of $X$ (\cref{lem:order}), then $a_n=\alpha n+\beta$
for all $n$.  Conversely, every such set is an $s$-equilibrium.
\end{theorem}

For $s=2$, \cref{thm:main} answers Benjamini's question, and Problem~1 of
\cite{GK17} for $F(d)=d^{-2}$, in the affirmative.  No a priori assumption is made
on the gaps: they may be unbounded, and they may accumulate at $0$.

\begin{remark}[Conventions and scope]\label{rem:scope}
(a) The note \cite{Ben15} does not say in which sense the series converges.
We read $\sum_{i\ne n}$ as a sum over an unordered index set, that is, as
an absolutely convergent series; this gives \eqref{eq:E1} and
\eqref{eq:E2}.  The arXiv version of \cite{GK17} defines equilibrium in the
same way: the total force on each particle from each side is finite, and
the net force is zero.  The published version states the finiteness
assumption for the force law instead (the one-sided forces are finite when
the particles sit at the integers) and defines equilibrium as zero net
force.  Its proofs compare one-sided forces, which therefore have to be
finite.  A ``principal value'' reading, in which the one-sided forces are
infinite and only symmetric partial sums cancel, is not covered by
\cref{thm:main}.

(b) We take $(a_n)$ to be the increasing enumeration of the configuration,
as \cite{GK17} do when they ask about the distances between subsequent
particles.  For an arbitrary bijection $\Z\to X$ the identity
$a_n=\alpha n+\beta$ fails trivially, already for a re-enumerated
arithmetic progression.

(c) For $s\le1$ an arithmetic progression violates \eqref{eq:E1}.  Problem~1
of \cite{GK17} for a general strictly decreasing force law remains open, and
we make no claim for exponents $s>2$ (\cref{rem:bigs,rem:generalF}).
\end{remark}

\paragraph{Idea of the proof.}
Let $g_n=a_{n+1}-a_n$.  Subtract the equilibrium equation at $a_{n+1}$ from
the one at $a_n$, after shifting the summation index by one.  This gives
$\sum_{j\ne n}c(n,j)(g_j-g_n)=0$ for explicit symmetric positive weights
$c(n,j)$ (\cref{lem:harmonic}).  For $s=2$ they are
$c(n,j)=(D+D')/(D^2D'^2)$ with $D=|a_j-a_n|$ and $D'=|a_{j+1}-a_{n+1}|$.
Thus $g$ is a positive harmonic function of a reversible random walk on
$\Z$ with long-range jumps.  In this language, \cite[Theorem~1]{GK17} is
the maximum principle (\cref{rem:max}).  A Liouville theorem would finish the
proof if the walk were known to be recurrent.  We prove recurrence for the
Doob transform of the walk by $g$, whose conductances are
$c(n,j)g_ng_j$, by exhibiting test functions of small energy
(\cref{prop:energy}).  The estimate uses a bound on the ratio of
consecutive gaps, which follows from equilibrium (\cref{lem:ratio}), and,
for $s=2$, a quadratic bound on the number of points in an interval, which
follows from \eqref{eq:E1}.  Then $1/g$, a positive harmonic function of the
recurrent transformed walk, is constant.  Since $g$ is harmonic, the
transformed walk is recurrent if and only if the original one is.  The
transform does not create recurrence; it supplies the test functions
(\cref{rem:doob}).

\section{The gaps form a harmonic function}\label{sec:gaps}

Throughout this section $s>1$, and $X$ is an $s$-equilibrium.

\begin{lemma}\label{lem:order}
The set $X$ is unbounded above and below.  Hence $X=\{a_n:n\in\Z\}$ for a
strictly increasing sequence with $a_n\to\pm\infty$ as $n\to\pm\infty$, and
this enumeration is unique up to a shift of the index.
\end{lemma}

\begin{proof}
Suppose that $X$ is bounded above.  Since $X$ is locally finite, it has a
largest element $x$, and the sum in \eqref{eq:E2} at $x$ equals
$-\sum_{y\in X,\,y<x}(x-y)^{-s}<0$, because $X$ has at least two points.
This contradicts \eqref{eq:E2}, and the same argument applies below.  A
locally finite set that is unbounded in both directions has exactly one
increasing enumeration by $\Z$ with a prescribed point at index $0$.
\end{proof}

From now on $(a_n)_{n\in\Z}$ is this enumeration, $g_n=a_{n+1}-a_n>0$ are
the gaps, and
\[
F_n^-=\sum_{i<n}(a_n-a_i)^{-s},\qquad F_n^+=\sum_{i>n}(a_i-a_n)^{-s}
\]
are the forces on $a_n$ from the left and from the right.  By \eqref{eq:E1}
and \eqref{eq:E2}, $F_n^-=F_n^+<\infty$ for every $n$.

\begin{lemma}[Ratios of consecutive gaps]\label{lem:ratio}
Let $r_s$ be the unique positive solution of $2r^{-s}+(1+r)^{-s}=1$.  Then
$r_s>1$, and
\[
r_s^{-1}g_{n-1}\le g_n\le r_s\,g_{n-1}\qquad\text{for all }n\in\Z .
\]
For $s=2$, $r_2=1.538615\ldots$
\end{lemma}

\begin{proof}
Fix $n$, and write $y=a_n$, $z=a_{n+1}$, $u=g_{n-1}$ and $v=g_n$.  For
$i\ge n+2$ we have $a_i-y>a_i-z>0$, hence
\[
F_n^+=v^{-s}+\sum_{i\ge n+2}(a_i-y)^{-s}\le v^{-s}+F_{n+1}^+ .
\]
The terms $i=n$ and $i=n-1$ of $F_{n+1}^-$ are $v^{-s}$ and $(u+v)^{-s}$,
and for $i\le n-2$ we have $z-a_i>y-a_i>0$.  Hence
\[
F_{n+1}^-\le v^{-s}+(u+v)^{-s}+\sum_{i\le n-2}(y-a_i)^{-s}
=v^{-s}+(u+v)^{-s}+F_n^--u^{-s}.
\]
Using $F_n^-=F_n^+$ and $F_{n+1}^-=F_{n+1}^+$, we obtain
\[
F_n^-=F_n^+\le v^{-s}+F_{n+1}^-\le2v^{-s}+(u+v)^{-s}+F_n^--u^{-s}.
\]
Since $F_n^-<\infty$, this gives $u^{-s}\le2v^{-s}+(u+v)^{-s}$, that is,
$h_s(v/u)\ge1$ with $h_s(r)=2r^{-s}+(1+r)^{-s}$.  The function $h_s$
decreases strictly from $+\infty$ to $0$ on $(0,\infty)$.  So $r_s$ is well
defined, $v/u\le r_s$, and $r_s>1$ because $h_s(1)>1$.

The reflected set $-X$ is again an $s$-equilibrium.  Its increasing
enumeration $b_m=-a_{-m}$ has the gaps $b_{m+1}-b_m=g_{-m-1}$.  The bound
just proved, applied to $-X$ at the index $m=-n$, gives
$g_{n-1}\le r_sg_n$.
\end{proof}

\begin{lemma}[The gaps are harmonic]\label{lem:harmonic}
For integers $n\ne j$ put $D=|a_j-a_n|$ and $D'=|a_{j+1}-a_{n+1}|$, and let
\[
c(n,j)=\frac{D^{-s}-D'^{-s}}{D'-D}\quad\text{if }D\ne D',\qquad
c(n,j)=sD^{-s-1}\quad\text{if }D=D'.
\]
Then $c(n,j)=c(j,n)>0$, and for every $n\in\Z$
\begin{equation}\label{eq:harmonic}
\sum_{j\ne n}c(n,j)\,(g_j-g_n)=0,
\end{equation}
where the series converges absolutely.  For $s=2$,
$c(n,j)=(D+D')/(D^2D'^2)$.
\end{lemma}

\begin{proof}
For $m\in\Z$ let $\Phi_m=\sum_{i\ne m}\sgn(a_i-a_m)|a_i-a_m|^{-s}$.  By
\eqref{eq:E1} and \eqref{eq:E2} this series converges absolutely, and
$\Phi_m=0$.  Substituting $i=j+1$ in $\Phi_{n+1}$ we get
\[
0=\Phi_n-\Phi_{n+1}=\sum_{j\ne n}\Bigl(\sgn(a_j-a_n)D^{-s}
-\sgn(a_{j+1}-a_{n+1})D'^{-s}\Bigr),
\]
where $D$ and $D'$ depend on $j$ as in the statement.  The subtraction may
be done termwise, because both series converge absolutely.  If $j>n$, both
signs are $+1$ and $D'-D=(a_{j+1}-a_j)-(a_{n+1}-a_n)=g_j-g_n$, so the
$j$-th term is $D^{-s}-D'^{-s}=c(n,j)(g_j-g_n)$.  If $j<n$, both signs are
$-1$ and $D-D'=(a_n-a_j)-(a_{n+1}-a_{j+1})=g_j-g_n$, so the $j$-th term is
$D'^{-s}-D^{-s}=c(n,j)(g_j-g_n)$.  (If $D=D'$, both expressions vanish.)
This proves \eqref{eq:harmonic}.

The numbers $D$ and $D'$ do not change when $n$ and $j$ are exchanged.
Moreover $c(n,j)$ is the divided difference of the strictly decreasing
function $t\mapsto t^{-s}$ at $D$ and $D'$, with the sign reversed.  Hence
$c(n,j)=c(j,n)>0$.  For $s=2$,
$(D^{-2}-D'^{-2})/(D'-D)=(D+D')/(D^2D'^2)$.
\end{proof}

\begin{lemma}[Comparison]\label{lem:comparison}
For $n\ne j$ let $D,D'$ be as in \cref{lem:harmonic}.  Then
$\min\{D,D'\}\ge D/(1+r_s)$ and
\begin{equation}\label{eq:cbound}
c(n,j)\le K_s\,|a_j-a_n|^{-s-1},\qquad K_s:=s(1+r_s)^{s+1}.
\end{equation}
For $s=2$ one also has $c(n,j)\le(1+r_2)(2+r_2)|a_j-a_n|^{-3}$, with
$(1+r_2)(2+r_2)=8.983\ldots$  Consequently
$\pi(n):=\sum_{j\ne n}c(n,j)<\infty$, and
$\sum_{j\ne n}c(n,j)\,g_j=\pi(n)\,g_n$ for every $n$.
\end{lemma}

\begin{proof}
By symmetry we may assume $j>n$.  If $j=n+1$, then $D=g_n$ and
$D'=g_{n+1}\ge g_n/r_s$ by \cref{lem:ratio}.  If $j\ge n+2$, then
$D\ge g_n+g_{n+1}\ge(1+r_s^{-1})g_n$, so $g_n\le r_sD/(1+r_s)$ and
$D'=D+g_j-g_n\ge D-g_n\ge D/(1+r_s)$.  In both cases
$\min\{D,D'\}\ge D/(1+r_s)$.  By the mean value theorem,
$c(n,j)=s\xi^{-s-1}$ for some $\xi$ between $D$ and $D'$, so
$c(n,j)\le s\min\{D,D'\}^{-s-1}\le K_sD^{-s-1}$.  For $s=2$,
$c(n,j)=D^{-1}D'^{-2}+D^{-2}D'^{-1}\le\bigl((1+r_2)^2+(1+r_2)\bigr)D^{-3}$.

Every $j\ne n$ satisfies $|a_j-a_n|\ge\delta_n:=\min\{g_{n-1},g_n\}>0$, so
$\pi(n)\le K_s\delta_n^{-1}(F_n^-+F_n^+)<\infty$.  Finally, the series
$\sum_jc(n,j)(g_j-g_n)$ converges absolutely with sum $0$ by
\cref{lem:harmonic}, and $\sum_jc(n,j)g_n=\pi(n)g_n$.  Adding the two
series gives the last identity.
\end{proof}

\begin{remark}[The maximum principle]\label{rem:max}
Identity \eqref{eq:harmonic} contains \cite[Theorem~1]{GK17} for the forces
$d^{-s}$.  If $g$ attains its maximum at $n$, all terms of
\eqref{eq:harmonic} are $\le0$, so they all vanish, and $c>0$ gives
$g_j=g_n$ for all $j$; the same holds for a minimum.  The proof of
\cref{lem:harmonic} uses only that the force law is strictly decreasing and
that the one-sided forces are finite (see \cref{rem:generalF}).
\end{remark}

\section{Reversible random walks with long-range jumps}\label{sec:networks}

Let $V$ be a countably infinite set, and let $c\colon V\times V\to[0,\infty)$
be symmetric, with $c(x,x)=0$ and
\[
\pi(x):=\sum_{y\in V}c(x,y)\in(0,\infty)\qquad\text{for all }x\in V.
\]
We call $c$ a \emph{network}.  A vertex may have infinitely many $y$ with
$c(x,y)>0$.  The random walk on $c$ is the Markov chain $(X_t)_{t\ge0}$ with
transition probabilities $P(x,y)=c(x,y)/\pi(x)$; it is reversible with
respect to $\pi$, and we always assume that it is irreducible.  We write
$P_x$ for the law of the walk started at $x$.  A function $u$ on $V$ is
\emph{harmonic} if $\sum_yP(x,y)|u(y)|<\infty$ and
$Pu(x):=\sum_yP(x,y)u(y)=u(x)$ for every $x$.  The walk, or the network, is
\emph{recurrent} if $P_x[X_t=x\text{ for some }t\ge1]=1$ for some, and then
every, $x$.  The \emph{energy} of $f\colon V\to\R$ is
$\Ec_c(f)=\frac12\sum_{x,y}c(x,y)(f(x)-f(y))^2\in[0,\infty]$.

For locally finite networks the three lemmas below are standard; see, for
example, \cite[Chapter~2]{LP16}, and \cite{Doo57} for the transform in
\cref{lem:doob}.  We include short proofs, because our networks are not
locally finite.  The only property used is $\pi<\infty$.

We use the following form of Green's identity.  If $u$ is finitely
supported and $\sum_yc(x,y)|w(y)|<\infty$ for every $x$, then
\begin{equation}\label{eq:green}
\frac12\sum_{x,y}c(x,y)\bigl(u(x)-u(y)\bigr)\bigl(w(x)-w(y)\bigr)
=\sum_xu(x)\sum_yc(x,y)\bigl(w(x)-w(y)\bigr),
\end{equation}
and all series converge absolutely.  Indeed, by the symmetry of $c$,
\[
\sum_{x,y}c(x,y)|u(x)-u(y)|\,|w(x)-w(y)|\le
2\sum_x|u(x)|\sum_yc(x,y)\bigl(|w(x)|+|w(y)|\bigr)<\infty .
\]
Expanding the product and exchanging $x$ and $y$ in one of the two resulting
sums gives \eqref{eq:green}.

\begin{lemma}[Dirichlet principle]\label{lem:dirichlet}
Let $x_0\in V$, and let $f\colon V\to\R$ be finitely supported with
$f(x_0)=1$.  Then
\[
\pi(x_0)\,P_{x_0}\bigl[X_t\ne x_0\text{ for all }t\ge1\bigr]\le\Ec_c(f).
\]
Consequently, if $\inf\Ec_c(f)=0$ over all such $f$, the walk is
recurrent.
\end{lemma}

\begin{proof}
Let $B=\{x_0\}\cup\operatorname{supp}f$ and $A=B\setminus\{x_0\}$, both
finite.  For $S\subset V$ let $\tau_S=\inf\{t\ge0:X_t\in S\}$ and
$\tau_S^+=\inf\{t\ge1:X_t\in S\}$.  By irreducibility, from every $x\in A$
the walk leaves $A$ within $|A|$ steps with positive probability.  Since
$A$ is finite, the Markov property shows that, from any starting point,
the walk stays in $A$ forever with probability $0$.

Let $h(x)=P_x[\tau_{\{x_0\}}<\tau_{V\setminus B}]$.  Then $0\le h\le1$,
$h(x_0)=1$, $h=0$ off $B$, and $Ph(x)=h(x)$ for $x\in A$ (first-step
analysis).  By the same analysis at $x_0$,
\[
q:=P_{x_0}\bigl[\tau^+_{V\setminus B}<\tau^+_{\{x_0\}}\bigr]
=1-Ph(x_0)=h(x_0)-Ph(x_0).
\]
By \eqref{eq:green} with $u=w=h$,
\[
\Ec_c(h)=\sum_{x\in B}h(x)\,\pi(x)\bigl(h(x)-Ph(x)\bigr)=\pi(x_0)\,q .
\]
The function $\eta=f-h$ vanishes at $x_0$ and off $B$, so it is supported
in $A$.  By \eqref{eq:green} with $u=\eta$ and $w=h$,
\[
\frac12\sum_{x,y}c(x,y)\bigl(h(x)-h(y)\bigr)\bigl(\eta(x)-\eta(y)\bigr)
=\sum_{x\in A}\eta(x)\,\pi(x)\bigl(h(x)-Ph(x)\bigr)=0 .
\]
All energies involved are finite, so
$\Ec_c(f)=\Ec_c(h)+\Ec_c(\eta)\ge\pi(x_0)q$.  Finally, if $X_t\ne x_0$ for
all $t\ge1$, then either the walk visits $V\setminus B$ before returning
to $x_0$, or it stays in $A$ from time $1$ on; the second event has
probability $0$.  Hence the probability of no return is at most
$q\le\Ec_c(f)/\pi(x_0)$.  If $\inf\Ec_c(f)=0$, then $x_0$ is recurrent, and
so is every vertex, by irreducibility.
\end{proof}

\begin{lemma}[Liouville property]\label{lem:liouville}
If the walk is recurrent, every positive harmonic function is constant.
\end{lemma}

\begin{proof}
Let $u>0$ be harmonic.  Under $P_x$ the process $u(X_t)$ is a nonnegative
martingale, so it converges almost surely to a finite limit.  By recurrence
and irreducibility, the walk visits every vertex infinitely often, almost
surely.  Hence $u(y)=u(z)$ for all $y,z\in V$.
\end{proof}

\begin{lemma}[Doob transform]\label{lem:doob}
Let $h\colon V\to(0,\infty)$ satisfy $\sum_yc(x,y)h(y)=\pi(x)h(x)$ for all
$x$, that is, let $h$ be a positive harmonic function.  Put
$c^h(x,y)=c(x,y)h(x)h(y)$.  Then the following hold.
\begin{enumerate}
\item[(a)] $\pi^h(x):=\sum_yc^h(x,y)=\pi(x)h(x)^2<\infty$, so $c^h$ is a
network.  Its walk is irreducible and has transition probabilities
$P^h(x,y)=P(x,y)h(y)/h(x)$.
\item[(b)] The function $1/h$ is harmonic for $P^h$.
\item[(c)] $(P^h)^t(x,y)=P^t(x,y)h(y)/h(x)$ for all $t\ge0$.  In
particular, the walks $P$ and $P^h$ have the same return probabilities, and
$c^h$ is recurrent if and only if $c$ is.
\item[(d)] $\Ec_c(h\varphi)=\Ec_{c^h}(\varphi)$ for every finitely supported
$\varphi\colon V\to\R$.
\end{enumerate}
\end{lemma}

\begin{proof}
(a) We have $\pi^h(x)=h(x)\sum_yc(x,y)h(y)=\pi(x)h(x)^2$, and
$c^h(x,y)/\pi^h(x)=P(x,y)h(y)/h(x)$.  (b) Indeed,
$\sum_yP^h(x,y)/h(y)=\sum_yP(x,y)/h(x)=1/h(x)$.  (c) The formula follows
from (a) by induction on $t$.  Along every path from $x$ back to $x$ the
factors $h(y)/h(x)$ cancel, so the two walks return to $x$ with the same
probability.  (d) Put $u=h\varphi^2$, which is finitely supported.  For
all $x,y$,
\[
\bigl(h(x)\varphi(x)-h(y)\varphi(y)\bigr)^2
=h(x)h(y)\bigl(\varphi(x)-\varphi(y)\bigr)^2
+\bigl(h(x)-h(y)\bigr)\bigl(u(x)-u(y)\bigr).
\]
Multiply by $c(x,y)/2$ and sum.  The first terms give $\Ec_{c^h}(\varphi)$,
which is finite because $\varphi$ is finitely supported and $\pi^h<\infty$.
By \eqref{eq:green} with $w=h$, the second terms are absolutely summable,
with sum $\sum_xu(x)\bigl(\pi(x)h(x)-\sum_yc(x,y)h(y)\bigr)=0$.
\end{proof}

\section{The energy estimate and the proof of the theorem}\label{sec:energy}

In this section $1<s\le2$, $X$ is an $s$-equilibrium, and we use the
notation of \cref{sec:gaps}.  We write $r=r_s$ and $K=K_s$.  By
\cref{lem:comparison}, $c$ (with $c(n,n):=0$) is a network on $\Z$ in the
sense of \cref{sec:networks}; its walk is irreducible because $c(n,j)>0$
for all $n\ne j$.  By \cref{lem:comparison}, $g$ is a positive harmonic function
for $c$.  Let
\[
c^*(n,j)=c(n,j)\,g_ng_j,\qquad \pi^*(n)=\pi(n)\,g_n^2 ,
\]
be the Doob transform of $c$ by $g$ (\cref{lem:doob}), and let $P^*$ be its
walk.  For $R>0$ let
\[
\varphi_R(x)=\min\bigl\{1,\max\{0,\,2-|x|/R\}\bigr\}\quad(x\in\R),
\qquad f_R(n)=\varphi_R(a_n)\quad(n\in\Z).
\]

\begin{proposition}\label{prop:energy}
Fix $R_1\ge|a_0|$ such that $(-2R_1,2R_1)$ contains at least two points of
$X$.  For $R\ge R_1$, the function $f_R$ is finitely supported,
$f_R(0)=1$, and
\begin{align*}
\Ec_{c^*}(f_R)&\le K\Bigl(\frac{7(1+r)}{2-s}+\frac{4(1+r)^{s+3}}{s}\Bigr)
R^{1-s}\qquad\text{if }s<2,\\
\Ec_{c^*}(f_R)&\le\frac KR\Bigl(7(1+r)\max\Bigl\{1,\,\log\frac{(1+r)
\bigl(1+16R^2(F_0^-+F_0^+)\bigr)}{7}\Bigr\}+2(1+r)^5\Bigr)
\qquad\text{if }s=2.
\end{align*}
In both cases $\Ec_{c^*}(f_R)\to0$ as $R\to\infty$.
\end{proposition}

\begin{proof}
Since $|a_0|\le R$, we have $f_R(0)=1$, and $f_R$ vanishes outside the
finite set $\{n:|a_n|<2R\}$.  By \eqref{eq:cbound},
\begin{equation}\label{eq:split}
\Ec_{c^*}(f_R)=\sum_{n<j}c(n,j)\,g_ng_j\bigl(f_R(n)-f_R(j)\bigr)^2
\le K(S_1+S_2),
\end{equation}
where $S_1$ is the sum of $g_ng_j(f_R(n)-f_R(j))^2(a_j-a_n)^{-s-1}$ over
the pairs $n<j$ with $a_j-a_n\le R$, and $S_2$ is the same sum over the
pairs $n<j$ with $a_j-a_n>R$.

\emph{Step 1: the cut-off.}  We have $0\le\varphi_R\le1$ and
$|\varphi_R(x)-\varphi_R(y)|\le|x-y|/R$.  Moreover, $\varphi_R$ is constant
on each of $(-\infty,-2R]$, $[-R,R]$ and $[2R,\infty)$, and it vanishes on
the first and the last.  Hence, if $n<j$ and $f_R(n)\ne f_R(j)$, then
$[a_n,a_j]$ meets $(-2R,-R)\cup(R,2R)$, and at least one of $a_n$, $a_j$
lies in $(-2R,2R)$.  If in addition $a_j-a_n\le R$, then both $a_n$ and
$a_j$ lie in $(-3R,3R)$.

\emph{Step 2: comparison with integrals.}  Let $p>0$ and $n<j$.  Then
\begin{align}
g_j\,(a_j-a_n)^{-p}&\le(1+r)^p\int_{a_j}^{a_{j+1}}(t-a_n)^{-p}\,dt,
\label{eq:int1}\\
g_n\,(a_j-a_n)^{-p}&\le r(1+r)^p\int_{a_{n-1}}^{a_n}(a_j-t)^{-p}\,dt.
\label{eq:int2}
\end{align}
For \eqref{eq:int1}, note that $g_j\le rg_{j-1}\le r(a_j-a_n)$ by
\cref{lem:ratio}, because $j-1\ge n$.  So $t-a_n\le a_{j+1}-a_n\le
(1+r)(a_j-a_n)$ for $t\in[a_j,a_{j+1}]$, and we integrate
$(t-a_n)^{-p}\ge(1+r)^{-p}(a_j-a_n)^{-p}$ over this interval of length
$g_j$.  For \eqref{eq:int2}, similarly, $g_{n-1}\le rg_n\le r(a_j-a_n)$,
so $a_j-t\le(1+r)(a_j-a_n)$ for $t\in[a_{n-1},a_n]$.  Integrating gives
$g_{n-1}(a_j-a_n)^{-p}\le(1+r)^p\int_{a_{n-1}}^{a_n}(a_j-t)^{-p}\,dt$, and
$g_n\le rg_{n-1}$.

\emph{Step 3: the near pairs.}  Let $I=\{n:|a_n|<3R,\ g_n\le R\}$, and for
$n\in I$ let $J(n)=\{j>n:a_j-a_n\le R\}$.  If a pair $n<j$ with
$a_j-a_n\le R$ contributes to $S_1$, then $|a_n|<3R$ by Step~1, and
$g_n\le a_j-a_n\le R$; so $n\in I$ and $j\in J(n)$.  Using
$(f_R(n)-f_R(j))^2\le(a_j-a_n)^2/R^2$ we obtain
\begin{equation}\label{eq:S1}
S_1\le R^{-2}\sum_{n\in I}g_n\sum_{j\in J(n)}g_j\,(a_j-a_n)^{1-s}.
\end{equation}
Let $n\in I$ with $J(n)\ne\emptyset$, and let $j^*=\max J(n)$.  The
intervals $[a_j,a_{j+1}]$, $j\in J(n)$, tile $[a_{n+1},a_{j^*+1}]$, and
$a_{j^*+1}-a_n\le(1+r)(a_{j^*}-a_n)\le(1+r)R$ by Step~2.  By
\eqref{eq:int1} with $p=s-1$ (and trivially if $J(n)=\emptyset$, since
$g_n\le R$),
\begin{equation}\label{eq:inner}
\sum_{j\in J(n)}g_j\,(a_j-a_n)^{1-s}\le(1+r)^{s-1}\int_{g_n}^{(1+r)R}
\tau^{1-s}\,d\tau
\le\begin{cases}
(1+r)R^{2-s}/(2-s), & s<2,\\
(1+r)\log\bigl((1+r)R/g_n\bigr), & s=2.
\end{cases}
\end{equation}
The intervals $[a_n,a_{n+1}]$, $n\in I$, are disjoint and lie in
$(-3R,4R)$, so $\sum_{n\in I}g_n\le7R$.  For $s<2$, \eqref{eq:S1} and
\eqref{eq:inner} give
\begin{equation}\label{eq:S1-s}
S_1\le\frac{7(1+r)}{2-s}\,R^{1-s}.
\end{equation}
For $s=2$, if $I=\emptyset$ then $S_1=0$; otherwise let $N=|I|\ge1$ and $\Sigma=\sum_{n\in I}g_n\le7R$.  The function
$\psi(t)=t\log((1+r)R/t)$ is concave on $(0,\infty)$, so Jensen's
inequality gives $\sum_{n\in I}\psi(g_n)\le N\psi(\Sigma/N)=
\Sigma\log\bigl((1+r)RN/\Sigma\bigr)$.  The function
$\Sigma\mapsto\Sigma\log(\Lambda/\Sigma)$ increases on $(0,\Lambda/e]$ and
is bounded by $\Lambda/e$.  With $\Lambda=(1+r)RN$ we get
$\sum_{n\in I}\psi(g_n)\le7R\max\{1,\log((1+r)N/7)\}$, and hence
\begin{equation}\label{eq:S1-2}
S_1\le\frac{7(1+r)}{R}\max\Bigl\{1,\log\frac{(1+r)N}{7}\Bigr\}
\qquad(s=2).
\end{equation}
Finally, if $|a_n|<3R$, then $|a_n-a_0|<4R$ because $|a_0|\le R$.  Every
$n\ne0$ with $|a_n-a_0|\le4R$ contributes at least $(4R)^{-s}$ to
$F_0^-+F_0^+$.  Hence
\begin{equation}\label{eq:count}
N\le\#\{n:|a_n|<3R\}\le1+(4R)^s\bigl(F_0^-+F_0^+\bigr),
\end{equation}
which is $1+16R^2(F_0^-+F_0^+)$ for $s=2$.

\emph{Step 4: the far pairs.}  By Step~1, every pair $n<j$ that contributes
to $S_2$ has $a_n\in(-2R,2R)$ or $a_j\in(-2R,2R)$.  We bound
$(f_R(n)-f_R(j))^2\le1$ and get
\[
S_2\le\sum_{|a_n|<2R}g_n\sum_{\substack{j>n\\a_j-a_n>R}}
\frac{g_j}{(a_j-a_n)^{s+1}}
+\sum_{|a_j|<2R}g_j\sum_{\substack{n<j\\a_j-a_n>R}}
\frac{g_n}{(a_j-a_n)^{s+1}} .
\]
For fixed $n$, the intervals $[a_j,a_{j+1}]$ with $a_j-a_n>R$ are disjoint
and lie in $[a_n+R,\infty)$.  So by \eqref{eq:int1} with $p=s+1$ the first
inner sum is at most $(1+r)^{s+1}\int_R^\infty\tau^{-s-1}d\tau=
(1+r)^{s+1}R^{-s}/s$.  For fixed $j$, the intervals $[a_{n-1},a_n]$ with
$a_j-a_n>R$ are disjoint and lie in $(-\infty,a_j-R]$.  So by
\eqref{eq:int2} the second inner sum is at most $r(1+r)^{s+1}R^{-s}/s$.
Let $m<M$ be the smallest and the largest index with $a_m,a_M\in(-2R,2R)$;
there are at least two such indices because $R\ge R_1$.  Then
$\sum_{|a_n|<2R}g_n=a_{M+1}-a_m=(a_M-a_m)+g_M$, where $a_M-a_m<4R$ and
$g_M\le rg_{M-1}\le r(a_M-a_m)$.  So $\sum_{|a_n|<2R}g_n\le4(1+r)R$, and
\begin{equation}\label{eq:S2}
S_2\le4(1+r)R\cdot\frac{(1+r)^{s+1}+r(1+r)^{s+1}}{s\,R^{s}}
=\frac{4(1+r)^{s+3}}{s}\,R^{1-s}.
\end{equation}
For $s=2$ this is $2(1+r)^5/R$.  Combining \eqref{eq:split},
\eqref{eq:S1-s}, \eqref{eq:S1-2}, \eqref{eq:count} and \eqref{eq:S2} proves
the proposition.
\end{proof}

\begin{proof}[Proof of \cref{thm:main}]
Let $X$ be an $s$-equilibrium with $1<s\le2$.  By \cref{lem:doob}, applied
to the network $c$ and its positive harmonic function $g$, the transform
$c^*$ is a network, its walk $P^*$ is irreducible, and $1/g$ is a positive
harmonic function for $P^*$.  By \cref{prop:energy} and
\cref{lem:dirichlet}, applied to $c^*$ with $x_0=0$ and $f=f_R$,
\[
\pi^*(0)\,P^*_0\bigl[X_t\ne0\text{ for all }t\ge1\bigr]\le\Ec_{c^*}(f_R)
\xrightarrow[R\to\infty]{}0 ,
\]
so $P^*$ is recurrent.  By \cref{lem:liouville}, $1/g$ is constant.  Hence
$g_n=g_0$ for all $n$, and $a_n=a_0+g_0n$.

Conversely, let $X=\{\alpha n+\beta:n\in\Z\}$ with $\alpha>0$.  For every
$x\in X$, both one-sided sums equal $\sum_{k\ge1}(\alpha k)^{-s}<\infty$,
since $s>1$.  So \eqref{eq:E1} and \eqref{eq:E2} hold.
\end{proof}

\section{Remarks}\label{sec:remarks}

\begin{remark}[The role of the Doob transform]\label{rem:doob}
By \cref{lem:doob}(c) with $h=g$, the walk $P^*$ has the same return
probabilities as the walk of $c$.  So, for an equilibrium, $c$ is recurrent
if and only if $c^*$ is: the transform does not turn a transient walk into
a recurrent one.  Its role is to supply test functions.  By
\cref{lem:doob}(d), the finitely supported functions $n\mapsto
g_n\varphi_R(a_n)/g_0$ take the value $1$ at $0$ and have $c$-energy
$\Ec_{c^*}(f_R)/g_0^2\to0$.  Thus the proof can also be phrased without
$P^*$: the network $c$ is recurrent by \cref{lem:dirichlet}, and $g$, a
positive harmonic function for $c$, is constant by \cref{lem:liouville}.
Heuristically, $c^*(n,j)\approx s\,g_ng_j|a_j-a_n|^{-1-s}$ is a Riemann
sum, with mesh given by the gaps, of the translation-invariant kernel
$s|x-y|^{-1-s}\,dx\,dy$ on $\R$.  This kernel does not see the local
density of the points; for $s=2$ the density enters only through the
logarithm in \eqref{eq:S1-2}.

\Cref{prop:energy} uses equilibrium only through \cref{lem:ratio} and, for
$s=2$, through $F_0^-+F_0^+<\infty$.  Its proof applies to every increasing
sequence whose consecutive gap ratios lie in $[r_s^{-1},r_s]$ and which,
when $s=2$, satisfies the point count \eqref{eq:count}.  For such a
sequence that is not an equilibrium, $g$ is in general not harmonic for the
weights $c$ of \cref{lem:harmonic}, $c^*$ is not a Doob transform of $c$,
and \cref{prop:energy} says nothing about the recurrence of $c$.
\end{remark}

\begin{remark}[What the proof uses]\label{rem:uses}
(a) Equilibrium of every particle is used in \cref{lem:ratio} and
\cref{lem:harmonic}.  If a single particle $a_k$ is not in equilibrium,
then \eqref{eq:harmonic} may fail at $n=k-1$ and $n=k$, and the argument
breaks down.  This is consistent with \cite[Theorem~7]{GK17}, which provides
non-trivial configurations in which all particles but one are in
equilibrium.

(b) The finiteness condition \eqref{eq:E1} is used in \cref{lem:harmonic}
and, for $s=2$, in the point count \eqref{eq:count}.  For $s=2$, any bound
with $\log\#\{n:|a_n|<3R\}=o(R)$ would suffice in place of
\eqref{eq:count}.

(c) Local finiteness, which follows from \eqref{eq:E1}, is used in
\cref{lem:order} and to make $f_R$ finitely supported.
\end{remark}

\begin{remark}[Exponents $s>2$]\label{rem:bigs}
\Cref{lem:order,lem:ratio,lem:harmonic,lem:comparison} and the bound
\eqref{eq:S2} hold for every $s>1$, with the same proofs.  For $s>2$,
however, the integral in \eqref{eq:inner} is of order $g_n^{2-s}$.  Then
the argument above controls $S_1$ only through $R^{-2}\sum_{n\in I}g_n^{3-s}$, which is large
when many gaps are small.  We make no claim for $s>2$; whether the
conclusion of \cref{thm:main} holds there remains open.
\end{remark}

\begin{remark}[Other force laws]\label{rem:generalF}
Let $F\colon(0,\infty)\to(0,\infty)$ be strictly decreasing, and assume
that the one-sided forces $\sum_{y\in X,\,y>x}F(y-x)$ and
$\sum_{y\in X,\,y<x}F(x-y)$ are finite and equal for every $x\in X$.  The
proof of \cref{lem:harmonic} then gives \eqref{eq:harmonic} with
$c(n,j)=(F(D)-F(D'))/(D'-D)>0$ when $D\ne D'$ (when $D=D'$ the term
vanishes, and any positive value of $c(n,j)$ will do).  So
\cref{rem:max} recovers \cite[Theorem~1]{GK17} for such configurations.
What is missing for general $F$ is recurrence: our estimates use the
homogeneity of $d^{-s}$ in \cref{lem:ratio,lem:comparison}.  Problem~1 of
\cite{GK17} for general $F$ remains open, and so does the principal-value
reading mentioned in \cref{rem:scope}(a).
\end{remark}

\paragraph{Related work.}
The question is a force-based variant of the crystallization problem; see
\cite{BL15} for a review.  In one dimension, Duneau and Katz \cite{DK84}
showed that the property of a two-body potential to give rise to a
one-dimensional lattice, defined through sequences of finite stable
equilibria, is stable under perturbations of the potential; they do not
classify infinite equilibria.  Since the pair potential $V(d)=d^{1-s}/(s-1)$ is convex, a
routine argument shows that an $s$-equilibrium minimizes the formal energy
$\sum_{i<j}V(a_j-a_i)$ with respect to compactly supported perturbations
that preserve the order of the points: by \eqref{eq:E1} the change of
energy is an absolutely convergent series, it is a convex function of the
perturbation on the convex set of order-preserving perturbations, and its
first variation vanishes by \eqref{eq:E2}.  In the present
translation-invariant setting, a configuration is ordered with respect to
all of its translates $(a_{n+k}+t)_{n\in\Z}$, $k\in\Z$, $t\in\R$ (meaning
that $a_{n+k}+t-a_n$ has the same sign for all $n$), if and only if it is
an arithmetic progression.  So \cref{thm:main} says that every $s$-equilibrium is
ordered with respect to all of its translates by index shifts and real
numbers.  This strengthens the Birkhoff property, which involves only
integer phase shifts, and links the problem to Aubry--Mather theory for
long-range interactions in statistical mechanics.
There, Candel and de la Llave \cite{CdlL98} and de la Llave and Valdinoci
\cite{dlLV10} construct ground states and critical points with the
Birkhoff property for every rotation number.  These are existence results,
under hypotheses that include convergent local energies and bounded
derivatives of the interaction (and finite range for the ground states of
\cite{CdlL98}), for configurations at bounded distance from a linear
configuration $n\mapsto\omega n$.  Neither paper claims that every ground
state in its general setting is ordered; Candel and de la Llave recall
that this holds for nearest-neighbour twist-map models on the line, by
the work of Aubry and Mather, and that it can fail in higher dimensions.  Their hypotheses do not cover the
interaction $d^{1-s}$ with $1<s\le2$, which is singular at $0$ and whose
local energies diverge at every configuration at bounded distance from a
linear one, nor configurations with unbounded gaps.  In our
setting existence is trivial, since arithmetic progressions are
equilibria; the content of \cref{thm:main} is that there are no other
equilibria.

\paragraph{Computational checks.}
Finite instances of the identities and inequalities of
\cref{sec:gaps,sec:networks,sec:energy} were also checked by computer, in
exact rational, high-precision or floating-point arithmetic, on random
and adversarial finite configurations and networks, and the bounds of
\cref{prop:energy} were evaluated numerically on several gap profiles.
These checks illustrate the statements; the proofs do not use them.  The
programs and their outputs are archived with the first public version of
this note, \href{https://doi.org/10.5281/zenodo.23113437}{doi:10.5281/zenodo.23113437}.

\paragraph{Scope.}  \Cref{thm:main} answers Question~0.1 of \cite{Ben15},
which is Question~5 in Section~4 of \cite{Syl15}, under the conventions of
\cref{rem:scope}.  It also answers Problem~1 of \cite{GK17} for the force
laws $d^{-s}$ with $1<s\le2$.  It does not cover general force laws,
exponents $s>2$, or a principal-value reading of the equilibrium
condition.  We are not aware of an earlier resolution of the question, or
of a partial result beyond \cite{GK17}.

\paragraph{Use of AI tools.}
AI-assisted tools were used to identify the problem and search the
literature, to find the proof, to write and run the computational checks,
and to help prepare this manuscript.  The author has reviewed the proofs,
the computations and the sources and is responsible for all claims.

\paragraph{Acknowledgment.}
The question was found in the UnsolvedMath collection of open problems
\cite{UM}.

\bibliographystyle{alphaurl}
\bibliography{references}

\end{document}